\documentclass[12pt,reqno]{amsart}
\usepackage{amsthm, amscd, amsfonts, amssymb, graphicx, color}
\usepackage[bookmarksnumbered, colorlinks, plainpages,linkcolor=green,anchorcolor=blue,citecolor=blue,urlcolor=blue]{hyperref}
\usepackage{enumerate}
\usepackage{exscale}
\usepackage[all]{xy}
\usepackage{relsize}
\usepackage{pdfsync}
\numberwithin{equation}{section}

\makeatletter
\@namedef{subjclassname@2020}{%
  \textup{2020} Mathematics Subject Classification}
\makeatother

\usepackage[includemp,body={398pt,550pt},footskip=30pt,%
            marginparwidth=60pt,marginparsep=10pt]{geometry}

\newtheorem{theorem}{Theorem}[section]
\newtheorem{lemma}[theorem]{Lemma}
\newtheorem{proposition}[theorem]{Proposition}
\newtheorem{corollary}[theorem]{Corollary}
\theoremstyle{definition}

\newtheorem{example}[theorem]{Example}

\numberwithin{equation}{section}

\newcommand{\R}{\mathbb{R}}

\newcommand{\C}{\mathbb{C}}

\newcommand{\Z}{\mathbb{Z}}
\newcommand{\s}{\mathcal{S}}
\newcommand{\res}{\mathrm{restricted}}

\begin{document}
\title[Toeplitz operators on weighted Fock spaces]{Toeplitz operators on weighted Fock spaces with $A_{\infty}$-type weights}
\date{December 8, 2025.
}
\author[J. Chen]{Jiale Chen}
\address{Jiale Chen, School of Mathematics and Statistics, Shaanxi Normal University, Xi'an 710119, China.}
\email{jialechen@snnu.edu.cn}

\thanks{This work was supported by National Natural Science Foundation of China (No. 12501170).}



\begin{abstract}
  \noindent By establishing some reproducing kernel estimates, we characterize the bounded, compact and Schatten $p$-class Toeplitz operators with positive measure symbols on the weighted Fock space $F^2_{\alpha,w}$ for $p\geq1$, where $w$ is a weight on the complex plane satisfying an $A_{\infty}$-type condition. Applications to Volterra operators and weighted composition operators are given.
\end{abstract}
\maketitle


\section{Introduction}
\allowdisplaybreaks[4]

Let $w$ be a weight, i.e. a non-negative and locally integrable function, on the complex plane $\C$. Given $0<p,\alpha<\infty$, we define the weighted space $L^p_{\alpha,w}$ as the collection of measurable functions $f$ on $\C$ such that
$$\|f\|^p_{L^p_{\alpha,w}}:=\int_{\C}|f(z)|^pe^{-\frac{p\alpha}{2}|z|^2}w(z)dA(z)<\infty,$$
where $dA$ is the Lebesgue measure on $\C$. The weighted Fock space $F^p_{\alpha,w}$ is defined to be the subspace of entire functions in $L^p_{\alpha,w}$ with the inherited (quasi-)norm. If $w\equiv\frac{\alpha}{\pi}$, then we obtain the weighted spaces $L^p_{\alpha}$ and the standard Fock spaces $F^p_{\alpha}$. We refer to \cite{Zh} for a brief account on Fock spaces.

It is well-known that for any $0<p,\alpha<\infty$, the Fock space $F^p_{\alpha}$ is closed in $L^p_{\alpha}$. Hence there is an orthogonal projection $P_{\alpha}$ from $L^2_{\alpha}$ onto $F^2_{\alpha}$, which is called the Fock projection and is given by
$$P_{\alpha}(f)(z):=\frac{\alpha}{\pi}\int_{\C}f(\xi)e^{\alpha z\overline{\xi}}e^{-\alpha|\xi|^2}dA(\xi),\quad z\in\C, \quad f\in L^2_{\alpha}.$$
To characterize the weighted boundedness of $P_{\alpha}$ on the spaces $L^p_{\alpha,w}$, Isralowitz \cite{Is} introduced the restricted $A_p$ weights. Here we use $Q$ to denote a square in $\C$ with sides parallel to the coordinate axes, and write $\ell(Q)$ for its side length. As usual, $p'$ denotes the conjugate exponent of $p$, i.e. $1/p+1/p'=1$, for $1<p<\infty$. Given $1<p<\infty$, a weight $w$ is said to belong to the class $A^{\res}_p$ if for some (or any) fixed $r>0$,
$$\sup_{Q:\ell(Q)=r}\left(\frac{1}{A(Q)}\int_{Q}wdA\right)
\left(\frac{1}{A(Q)}\int_{Q}w^{-\frac{p'}{p}}dA\right)^{\frac{p}{p'}}<\infty,$$
and $w$ is said to belong to the class $A^{\res}_1$ if for some (or any) fixed $r>0$,
$$\sup_{Q:\ell(Q)=r}\left(\frac{1}{A(Q)}\int_{Q}wdA\right)\left\|w^{-1}\right\|_{L^{\infty}(Q)}<\infty.$$
It was proved in \cite[Theorem 3.1]{Is} and \cite[Proposition 2.7]{CFP} that for $1\leq p<\infty$, $P_{\alpha}$ is bounded on the weighted space $L^p_{\alpha,w}$ if and only if $w\in A^{\res}_p$. Similar to the Muckenhoupt weights, we write
$$A^{\res}_{\infty}:=\bigcup_{1\leq p<\infty}A^{\res}_p.$$
Recently, the function and operator theory on weighted Fock spaces induced by weights from $A^{\res}_{\infty}$ developed quickly; see \cite{CFP23,CFP,Ch24,Ch24-1,Ch25,CHW,CW24,Xu}.

The theory of Toeplitz operators on standard Fock spaces has drawn lots of attention; see \cite{BCI,CIL,Fu,HL,IZ,Me,WZ} and the references therein. The aim of this paper is to investigate the basic properties of Toeplitz operators acting on weighted Fock spaces $F^2_{\alpha,w}$ with $w\in A^{\res}_{\infty}$. For $z\in \C$, let $B^w_z$ denote the reproducing kernel of $F^2_{\alpha,w}$ at $z$. Then for any $f\in F^2_{\alpha,w}$,
$$f(z)=\langle f,B^w_z\rangle_{F^2_{\alpha,w}}:=\int_{\C}f(\xi)\overline{B^w_z(\xi)}e^{-\alpha|\xi|^2}w(\xi)dA(\xi).$$
Given a positive Borel measure $\mu$ on $\C$, the Toeplitz operator $T_{\mu}$ is formally defined for entire functions $f$ on $\C$ by
$$T_{\mu}f(z):=\int_{\C}f(\xi)\overline{B^w_z(\xi)}e^{-\alpha|\xi|^2}d\mu(\xi),\quad z\in\C.$$
In this paper, we consider the boundedness, compactness and membership in Schatten $p$-classes of Toeplitz operators $T_{\mu}$ on the weighted Fock spaces $F^2_{\alpha,w}$ induced by $w\in A^{\res}_{\infty}$.

For any $\gamma\in\R$, it is easy to see that the weight $w_{\gamma}(z):=(1+|z|)^{\gamma}$ belongs to $A^{\res}_{\infty}$ (see \cite[Lemma 2.1]{CFP}). Hence the weighted Fock spaces $F^p_{\alpha,w}$ induced by weights from $A^{\res}_{\infty}$ contain the Fock--Sobolev spaces introduced in \cite{CCK,CZ} as a special case, and the main results of this paper generalize the corresponding results from \cite{CIJ,WCX}.

To state our main results, we need some notions. Let $D(z,r)$ be the disk centered at $z\in\C$ with radius $r>0$. Then the average function $\widehat{\mu}_{w,r}$ is defined by
$$\widehat{\mu}_{w,r}(z):=\frac{\mu(D(z,r))}{w(D(z,r))},\quad z\in\C.$$
Here and in the sequel, we write $w(E):=\int_EwdA$ for Borel subset $E\subset\C$. Another important tool in the theory of Toeplitz operators is the Berezin transform. For $z\in\C$, let $b^w_z:=B^w_z/\|B^w_z\|_{F^2_{\alpha,w}}$ be the normalized reproducing kernel. The Berezin transform $\widetilde{\mu}$ of the positive Borel measure $\mu$ is defined by  
$$\widetilde{\mu}(z):=\int_{\C}|b^w_z(\xi)|^2e^{-\alpha|\xi|^2}d\mu(\xi),\quad z\in\C.$$
We are now ready to state our first result, which characterizes the boundedness of $T_{\mu}$ on $F^2_{\alpha,w}$.

\begin{theorem}\label{bdd}
Let $\alpha>0$, $w\in A^{\res}_{\infty}$, and let $\mu$ be a positive Borel measure on $\C$. Then there exists $\delta\in(0,1)$ such that the following conditions are equivalent:
\begin{enumerate}
	\item [(a)] $T_{\mu}$ is bounded on $F^2_{\alpha,w}$;
	\item [(b)] $\widetilde{\mu}$ is bounded on $\C$;
	\item [(c)] for some (or any) $r\in(0,\delta)$, $\widehat{\mu}_{w,r}$ is bounded on $\C$.
\end{enumerate}
Moreover,
$$\|T_{\mu}\|_{F^2_{\alpha,w}\to F^2_{\alpha,w}}\asymp\sup_{z\in\C}\widetilde{\mu}(z)\asymp\sup_{z\in\C}\widehat{\mu}_{w,r}(z).$$
\end{theorem}

Our next result concerns the essential norm estimate of $T_{\mu}$. Recall that for a bounded linear operator $T$ on a Hilbert space $H$, the essential norm of $T$ is defined by
$$\|T\|_{e,H\to H}:=\inf_{K\in\mathcal{K}(H)}\|T-K\|_{H\to H},$$
where $\mathcal{K}(H)$ denotes the algebra of compact operators on $H$. It is clear that $T$ is compact if and only if $\|T\|_{e,H\to H}=0$.

\begin{theorem}\label{ess-norm}
Let $\alpha>0$, $w\in A^{\res}_{\infty}$, and let $\mu$ be a positive Borel measure on $\C$ such that $T_{\mu}$ is bounded on $F^2_{\alpha,w}$. Then there exists $\delta\in(0,1)$ such that for $r\in(0,\delta)$,
$$\|T_{\mu}\|_{e,F^2_{\alpha,w}\to F^2_{\alpha,w}}
\asymp\limsup_{|z|\to\infty}\widetilde{\mu}(z)\asymp\limsup_{|z|\to\infty}\widehat{\mu}_{w,r}(z).$$
\end{theorem}

As an immediate consequence, we have the following description for the compactness of $T_{\mu}$.

\begin{corollary}\label{cpt}
Let $\alpha>0$, $w\in A^{\res}_{\infty}$, and let $\mu$ be a positive Borel measure on $\C$. Then there exists $\delta\in(0,1)$ such that the following conditions are equivalent:
\begin{enumerate}
	\item [(a)] $T_{\mu}$ is compact on $F^2_{\alpha,w}$;
	\item [(b)] $\widetilde{\mu}(z)\to0$ as $|z|\to\infty$;
	\item [(c)] for some (or any) $r\in(0,\delta)$, $\widehat{\mu}_{w,r}(z)\to0$ as $|z|\to\infty$.
\end{enumerate}
\end{corollary}

Let $T$ be a compact operator on a separable Hilbert space $H$. Then there exist orthonormal sets $\{\sigma_n\}$, $\{e_n\}$ in $H$ and a non-increasing sequence $\{s_n(T)\}$ of non-negative numbers tending to $0$ such that for all $x\in H$,
$$Tx=\sum_{n\geq1}s_n(T)\langle x,e_n\rangle_{H}\sigma_n.$$
This is the canonical decomposition of $T$ and $s_n(T)$ is called the $n$th singular value of $T$. For $p>0$, the operator $T$ is said to be in the Schatten $p$-class $\s_p(H)$ if
$$\|T\|^p_{\s_p(H)}:=\sum_{n\geq1}s_n(T)^p<\infty.$$
We refer to \cite[Chapter 1]{Zhudisk} for a brief account on Schatten classes.

Our third result is the following characterization of the Schatten $p$-class Toeplitz operators $T_{\mu}$ on $F^2_{\alpha,w}$ for $p\geq1$.

\begin{theorem}\label{schatten-p}
Let $\alpha>0$, $w\in A^{\res}_{\infty}$, and let $\mu$ be a positive Borel measure on $\C$. Then there exists $\delta\in(0,1)$ such that for $p\geq1$ the following conditions are equivalent:
\begin{enumerate}
	\item [(a)] $T_{\mu}\in \s_{p}(F^2_{\alpha,w})$;
	\item [(b)] $\widetilde{\mu}\in L^p(\C,dA)$;
	\item [(c)] for some (or any) $r\in(0,\delta)$, $\widehat{\mu}_{w,r}\in L^p(\C,dA)$.
\end{enumerate}
Moreover,
$$\|T_{\mu}\|_{\s_p(F^2_{\alpha,w})}\asymp\|\widetilde{\mu}\|_{L^p(\C,dA)}\asymp\|\widehat{\mu}_{w,r}\|_{L^p(\C,dA)}.$$
\end{theorem}

The main obstacle to prove our results is the lack of explicit expression for the reproducing kernels $B^w_a$. By adapting a method from \cite{LR}, we establish the following local pointwise estimate for $B^w_a$ from below (see Theorem \ref{pointwise-below}): there exists $\delta\in(0,1)$ such that for any $a,z\in\C$ with $|z-a|<\delta$,
$$|B^w_a(z)|\gtrsim\left\|B^w_a\right\|_{F^2_{\alpha,w}}\cdot\left\|B^w_z\right\|_{F^2_{\alpha,w}},$$
which plays an essential role in the proof of the main results.

The rest part of this paper is organized as follows. Some preliminaries are given in Section \ref{pre}. Section \ref{rke} is devoted to establishing some estimates for the reproducing kernels $B^w_z$. In Section \ref{bdd-cpt} we prove Theorems \ref{bdd} and \ref{ess-norm}, while Section \ref{sp} contains the proof of Theorem \ref{schatten-p}. Finally, in Section \ref{app} we give some applications of the main results to Volterra operators and weighted composition operators.

Throughout the paper, the notation $A\lesssim B$ (or $B\gtrsim A$) means that there exists a nonessential constant $C>0$ such that $A\leq CB$. If $A\lesssim B\lesssim A$, then we write $A\asymp B$.

\section{Preliminaries}\label{pre}

In this section, we give some preliminary results that will be used in the sequel.

We first recall the following estimates on $A_{\infty}$-type weights. Here, for any $r>0$, we treat $r\Z^{2}$ as a subset of $\C$ in the natural way. For $z\in\C$ and $r>0$, we write $Q_r(z)$ to denote the square centered at $z$ with side length $\ell(Q)=r$.

\begin{lemma}\label{esti}
Let $w\in A^{\mathrm{restricted}}_{\infty}$ and $r>0$.
\begin{enumerate}
	\item [(1)] There exists $C>1$ such that for any $\nu,\nu'\in r\Z^{2}$,
	\begin{equation}\label{ss}
	\frac{w(Q_r(\nu))}{w(Q_r(\nu'))}\leq C^{|\nu-\nu'|}.
	\end{equation}
	\item [(2)] For any fixed $M,N\geq1$,
	\begin{equation}\label{sd}
		w(Q_r(z))\asymp w(Q_{Nr}(u))\asymp w(D(z,r))\asymp w(D(u,Nr))
	\end{equation}
	whenever $z,u\in\C$ satisfy $|z-u|<Mr$.
	\item [(3)] For any $\alpha>0$,
	\begin{equation}\label{int}
		\int_{\C}e^{-\alpha|z|^2}w(z)dA(z)<\infty.
	\end{equation}
\end{enumerate}
\end{lemma}
\begin{proof}
See \cite[Lemma 3.4]{Is} and \cite[Remark 2.3, Lemma 2.8]{CFP}.
\end{proof}

It follows from \eqref{ss} and \eqref{sd} that, if $w\in A^{\res}_{\infty}$, then for any $r>0$, there exists $C>1$ such that for any $z,u\in\C$,
\begin{equation}\label{dd}
w(D(z,r))\lesssim C^{|z-u|}w(D(u,r))
\end{equation}
with implicit constant depending only on $w$ and $r$.

The following lemma establishes some pointwise estimates for entire functions, which can be found in \cite[Lemma 3.1]{CFP}.

\begin{lemma}\label{pointwise}
Let $\alpha,p,r>0$, $w\in A^{\res}_{\infty}$, and let $f$ be an entire function on $\C$. Then for any $z\in\C$,
$$|f(z)|^pe^{-\frac{p\alpha}{2}|z|^2}\lesssim
\frac{1}{w(D(z,r))}\int_{D(z,r)}|f(\xi)|^pe^{-\frac{p\alpha}{2}|\xi|^2}w(\xi)dA(\xi),$$
where the implicit constant is independent of $f$ and $z$.
\end{lemma}

We next recall some equivalent norms for the spaces $F^p_{\alpha,w}$. To this end, for $r>0$, let $\widehat{w}_r$ be the weight defined by
$$\widehat{w}_r(z):=w(D(z,r)),\quad z\in\C.$$
The following lemma was established in \cite[Lemma 3.2]{CHW}.

\begin{lemma}\label{hat-norm}
Let $p,\alpha>0$ and $w\in A^{\res}_{\infty}$. Then for any $r>0$, $F^p_{\alpha,w}=F^p_{\alpha,\widehat{w}_r}$ with equivalent norms.
\end{lemma}

%
%
%
%

\section{Reproducing kernel estimates}\label{rke}

The purpose of this section is to establish some estimates for the reproducing kernels $B^w_z$ of the spaces $F^2_{\alpha,w}$ and show the weak convergence of normalized reproducing kernels. We begin with the following norm estimate.

\begin{lemma}\label{norm}
Let $\alpha,r>0$ and $w\in A_{\infty}^{\res}$. Then for $a\in\C$,
$$B^w_a(a)=\|B^w_a\|^2_{F^2_{\alpha,w}}\asymp\frac{e^{\alpha|a|^2}}{w(D(a,r))},$$
where the implicit constant is independent of $a$.
\end{lemma}
\begin{proof}
Let $L_a$ be the point evaluation at $a$ on $F^2_{\alpha,w}$. Then it is well-known that
\begin{equation}\label{delta_a}
\|B^w_a\|_{F^2_{\alpha,w}}=\|L_a\|_{F^2_{\alpha,w}\to\C}.
\end{equation}
Hence the upper estimate follows from Lemma \ref{pointwise}. To establish the lower estimate, denote $K_a(z)=e^{\alpha\overline{a}z}$. Then by \cite[Proposition 4.1]{CFP},
$$\|K_{a}\|_{F^2_{\alpha,w}}\asymp e^{\frac{\alpha}{2}|a|^2}w(D(a,r))^{1/2}.$$
Consequently,
$$\|L_a\|_{F^2_{\alpha,w}\to\C}\geq\frac{e^{\alpha|a|^2}}{\|K_a\|_{F^2_{\alpha,w}}}
\asymp\frac{e^{\frac{\alpha}{2}|a|^2}}{w(D(a,r))^{1/2}},$$
which finishes the proof.
\end{proof}

Based on Lemma \ref{norm} and the Cauchy--Schwarz inequality, we obtain that for $a,z\in\C$,
\begin{equation}\label{pu}
|B^w_a(z)|\lesssim\frac{e^{\frac{\alpha}{2}|z|^2}}{w(D(z,r))^{1/2}}\cdot\frac{e^{\frac{\alpha}{2}|a|^2}}{w(D(a,r))^{1/2}}
\end{equation}
with implicit constant independent of $a$ and $z$.


The following theorem establishes a local pointwise estimate for the reproducing kernels $B^w_z$ from below, which plays an essential role in the proofs of the main results. Our method is adapted from \cite[Lemma 3.6]{LR}.

\begin{theorem}\label{pointwise-below}
Let $\alpha,r>0$ and $w\in A_{\infty}^{\res}$. Then there exists $\delta=\delta(\alpha,w)\in(0,1)$ such that for $a\in\C$ and $z\in D(a,\delta)$,
$$|B^w_a(z)|\gtrsim\frac{e^{\frac{\alpha}{2}|a|^2+\frac{\alpha}{2}|z|^2}}{w(D(a,r))},$$
where the implicit constant is independent of $a$ and $z$.
\end{theorem}

To prove the above theorem, for each fixed $a\in\C$ we consider the subspace $F^2_{\alpha,w}(a)$ of $F^2_{\alpha,w}$ defined by
$$F^2_{\alpha,w}(a):=\{f\in F^2_{\alpha,w}:f(a)=0\}.$$
Let $\mathcal{V}_a$ be the one-dimensional subspace spanned by the reproducing kernel $B^w_a$. Then, noting that for any $f\in F^2_{\alpha,w}$,
$$f=f-\frac{f(a)}{B^w_a(a)}B^w_a+\frac{f(a)}{B^w_a(a)}B^w_a,$$
we have
$$F^2_{\alpha,w}=F^2_{\alpha,w}(a)\oplus\mathcal{V}_a.$$
Let the operator $S_a$ be defined for $f\in F^2_{\alpha,w}(a)$ by
$$S_af(z):=\frac{f(z)}{z-a},\quad z\in\C.$$
We have the following lemma.

\begin{lemma}\label{Sa-bdd}
Let $\alpha>0$ and $w\in A^{\res}_{\infty}$. The operator $S_a$ is bounded from $F^2_{\alpha,w}(a)$ into $F^2_{\alpha,w}$.
\end{lemma}
\begin{proof}
Fix $f\in F^2_{\alpha,w}(a)$. Since $f(a)=0$, there exist $\epsilon>0$ and an analytic function $g$ on $D(a,\epsilon)$ such that $f(z)=(z-a)g(z)$ for $z\in D(a,\epsilon)$. Then we have
\begin{align*}
&\int_{\C}|S_af(z)|^2e^{-\alpha|z|^2}w(z)dA(z)\\
&\ =\int_{\C\setminus D(a,\epsilon/2)}\left|\frac{f(z)}{z-a}\right|^2e^{-\alpha|z|^2}w(z)dA(z)
    +\int_{D(a,\epsilon/2)}|g(z)|^2e^{-\alpha|z|^2}w(z)dA(z)\\
&\ \leq\frac{4}{\epsilon^2}\|f\|^2_{F^2_{\alpha,w}}
    +\sup_{z\in D(a,\epsilon/2)}|g(z)|^2\cdot w(D(a,\epsilon/2))<\infty,
\end{align*}
which gives that $S_af\in F^2_{\alpha,w}$. Now, for any positive integer $k$,
\begin{align*}
\|S_af\|^2_{F^2_{\alpha,w}}&=\left(\int_{D(a,1/k)}+\int_{\C\setminus D(a,1/k)}\right)|S_af(z)|^2e^{-\alpha|z|^2}w(z)dA(z)\\
&=:I_1(k)+I_2(k).
\end{align*}
It follows from the Cauchy--Schwarz inequality, Lemma \ref{norm} and the estimate \eqref{sd} that
\begin{align*}
I_1(k)&=\int_{D(a,1/k)}|S_af(z)|^2e^{-\alpha|z|^2}w(z)dA(z)\\
&\leq\int_{D(a,1/k)}\|S_af\|^2_{F^2_{\alpha,w}}\|B^w_z\|^2_{F^2_{\alpha,w}}e^{-\alpha|z|^2}w(z)dA(z)\\
&\asymp\|S_af\|^2_{F^2_{\alpha,w}}\int_{D(a,1/k)}\frac{w(z)}{w(D(z,1))}dA(z)\\
&\asymp\frac{w(D(a,1/k))}{w(D(a,1))}\|S_af\|^2_{F^2_{\alpha,w}}.
\end{align*}
Hence we can choose a sufficiently large $k$, depending only on $\alpha$ and $w$, such that $I_1(k)\leq\frac{1}{2}\|S_af\|^2_{F^2_{\alpha,w}}$. Consequently,
\begin{align*}
\|S_af\|^2_{F^2_{\alpha,w}}&\leq2\int_{\C\setminus D(a,1/k)}|S_af(z)|^2e^{-\alpha|z|^2}w(z)dA(z)\\
&=2\int_{\C\setminus D(a,1/k)}\left|\frac{f(z)}{z-a}\right|^2e^{-\alpha|z|^2}w(z)dA(z)\\
&\leq2k^2\|f\|^2_{F^2_{\alpha,w}}.
\end{align*}
Since $f\in F^2_{\alpha,w}(a)$ is arbitrary, we conclude that $S_a$ is bounded from $F^2_{\alpha,w}(a)$ into $F^2_{\alpha,w}$.
\end{proof}

We are now ready to prove Theorem \ref{pointwise-below}.
\begin{proof}[Proof of Theorem \ref{pointwise-below}]
Fix $a\in\C$ and let $B^{w,a}_z$ be the reproducing kernel of $F^2_{\alpha,w}(a)$ at $z\in\C$. Then the point evaluation $L^a_z$ on $F^2_{\alpha,w}(a)$ satisfies
\begin{equation}\label{Laz}
\|L^a_z\|_{F^2_{\alpha,w}(a)\to\C}=\|B^{w,a}_z\|_{F^2_{\alpha,w}}
\end{equation}
and
$$L^a_z=(z-a)L_zS_a,$$
where $L_z$ is the point evaluation on $F^2_{\alpha,w}$. In view of Lemma \ref{Sa-bdd}, we choose $\delta=\frac{1}{2\|S_a\|_{F^2_{\alpha,w}(a)\to F^2_{\alpha,w}}}$. Consequently, for $z\in D(a,\delta)$,
$$\|L^a_z\|_{F^2_{\alpha,w}(a)\to\C}\leq|z-a|\cdot\|L_z\|_{F^2_{\alpha,w}\to\C}\cdot\|S_a\|_{F^2_{\alpha,w}(a)\to F^2_{\alpha,w}}
\leq\frac{1}{2}\|L_z\|_{F^2_{\alpha,w}\to\C},$$
which, combined with \eqref{delta_a} and \eqref{Laz}, implies that
\begin{equation}\label{1/2}
\|B^{w,a}_z\|_{F^2_{\alpha,w}}\leq\frac{1}{2}\|B^w_z\|_{F^2_{\alpha,w}}.
\end{equation}
On the other hand, for any $f\in F^2_{\alpha,w}(a)\subset F^2_{\alpha,w}$,
$$\langle f,B^{w,a}_z\rangle_{F^2_{\alpha,w}}=f(z)=\langle f,B^w_z\rangle_{F^2_{\alpha,w}}
=\left\langle f,B^w_z-\frac{B^w_z(a)}{B^w_a(a)}B^w_a\right\rangle_{F^2_{\alpha,w}}.$$
Hence
$$B^{w,a}_z=B^w_z-\frac{B^w_z(a)}{B^w_a(a)}B^w_a.$$
This, together with \eqref{1/2}, gives that for $z\in D(a,\delta)$,
$$\frac{1}{4}\|B^w_z\|^2_{F^2_{\alpha,w}}\geq\|B^{w,a}_z\|^2_{F^2_{\alpha,w}}
=B^{w,a}_z(z)=B^w_z(z)-\frac{|B^w_{a}(z)|^2}{B^w_a(a)},$$
which, in conjunction with Lemma \ref{norm} and the inequality \eqref{sd}, yields that
$$|B^w_a(z)|\gtrsim\|B^w_a\|_{F^2_{\alpha,w}}\|B^w_z\|_{F^2_{\alpha,w}}
\asymp\frac{e^{\frac{\alpha}{2}|a|^2+\frac{\alpha}{2}|z|^2}}{w(D(a,r))}.$$
The proof is complete.
\end{proof}

It was proved in \cite[Theorem 3.4]{CW24} that for $w\in A^{\res}_2$, polynomials are dense in $F^2_{\alpha,w}$. The following lemma indicates that for all $p>0$ and $w\in A^{\res}_{\infty}$, polynomials are dense in $F^p_{\alpha,w}$.

\begin{lemma}\label{dense}
Let $p,\alpha>0$ and $w\in A^{\res}_{\infty}$. Suppose $f\in F^p_{\alpha,w}$, and denote $f_r(z):=f(rz)$ for $r\in(0,1)$. Then
\begin{enumerate}
	\item [(i)] $f_r\to f$ in $F^p_{\alpha,w}$ as $r\to1^-$;
	\item [(ii)] there exists a sequence $\{p_n\}$ of polynomials that converges to $f$ in $F^p_{\alpha,w}$.
\end{enumerate} 
\end{lemma}
\begin{proof}
(i) By Lemma \ref{hat-norm}, $f\in F^p_{\alpha,\widehat{w}_1}$, and it suffices to show that $f_r\to f$ in $F^p_{\alpha,\widehat{w}_1}$. For any $r\in(0,1)$,
\begin{align*}
\|f_r\|^p_{F^p_{\alpha,\widehat{w}_1}}&=\int_{\C}|f(rz)|^pe^{-\frac{p\alpha}{2}|z|^2}\widehat{w}_1(z)dA(z)\\
&=\frac{1}{r^2}\int_{\C}|f(z)|^pe^{-\frac{p\alpha}{2}|z|^2}\widehat{w}_1(z)\cdot
    e^{-\frac{p\alpha}{2}|z|^2(r^{-2}-1)}\frac{\widehat{w}_1(z/r)}{\widehat{w}_1(z)}dA(z).
\end{align*}
It follows from \eqref{dd} that there exists a constant $C>1$ such that for any $z\in\C$ and $r\in(0,1)$,
$$\frac{\widehat{w}_1(z/r)}{\widehat{w}_1(z)}\lesssim C^{|z|(r^{-1}-1)}.$$
Consequently,
$$e^{-\frac{p\alpha}{2}|z|^2(r^{-2}-1)}\frac{\widehat{w}_1(z/r)}{\widehat{w}_1(z)}
\lesssim C^{|z|(r^{-1}-1)}e^{-\frac{p\alpha}{2}|z|^2(r^{-2}-1)}
\leq C^{\frac{1-r}{2p\alpha(1+r)}\log C}\leq C^{\frac{1}{2p\alpha}\log C},$$
and so the dominated convergence theorem yields $\|f_r\|_{F^p_{\alpha,\widehat{w}_1}}\to \|f\|_{F^p_{\alpha,\widehat{w}_1}}$ as $r\to1^-$. This together with the fact that $f_r\to f$ pointwisely as $r\to1^-$ gives the desired result (see for instance \cite[Lemma 3.17]{HKZ}).

(ii) We finish the proof by showing that for every $r\in (0,1)$, $f_r$ can be approximated by its Taylor polynomials in $F^p_{\alpha,w}$. To this end, fix $r\in(0,1)$ and $\beta\in(\alpha r^2,\alpha)$. Then by Lemma \ref{pointwise} and the inequality \eqref{dd},
\begin{align*}
\int_{\C}|f_r(z)|^2e^{-\beta|z|^2}dA(z)
&\lesssim\|f\|^2_{F^p_{\alpha,w}}\int_{\C}\frac{e^{-\alpha |rz|^2}}{w(D(rz,1))^{2/p}}e^{-\beta|z|^2}dA(z)\\
&\lesssim\frac{\|f\|^2_{F^p_{\alpha,w}}}{w(D(0,1))^{2/p}}\int_{\C}C^{r|z|}e^{-(\beta-\alpha r^2)|z|^2}dA(z)<\infty.
\end{align*}
Hence $f_r\in F^2_{\beta}$. Similarly, we can establish the bounded embedding $F^2_{\beta}\subset F^p_{\alpha,w}$. Therefore, if $p_n$ is the $n$th Taylor polynomial of $f_r$, then we have
$$\|f_r-p_n\|_{F^p_{\alpha,w}}\lesssim\|f_r-p_n\|_{F^2_{\beta}}\to0$$
as $n\to\infty$. The proof is complete.
\end{proof}

We end this section by the following proposition, which establishes the weak convergence of normalized reproducing kernels.

\begin{proposition}\label{weakly0}
Let $\alpha>0$ and $w\in A^{\res}_{\infty}$. Then $b^w_z$ converges to $0$ weakly in $F^2_{\alpha,w}$ as $|z|\to\infty$.
\end{proposition}
\begin{proof}
Using Lemma \ref{norm} and the inequality \eqref{dd}, we have for any polynomial $g$,
\begin{align*}
|\langle g,b^w_z\rangle_{F^2_{\alpha,w}}|&=\frac{|g(z)|}{\|B^w_z\|_{F^2_{\alpha,w}}}\\
&\asymp w(D(z,1))^{1/2}\cdot|g(z)|e^{-\frac{\alpha}{2}|z|^2}\\
&\lesssim w(D(0,1))^{1/2}\cdot C^{\frac{1}{2}|z|}|g(z)|e^{-\frac{\alpha}{2}|z|^2}\to0
\end{align*}
as $|z|\to\infty$. By Lemma \ref{dense}, polynomials are dense in $F^2_{\alpha,w}$, so we obtain that $b^w_z\to0$ weakly in $F^2_{\alpha,w}$ as $|z|\to\infty$.
\end{proof}

\section{Bounded and compact Toeplitz operators}\label{bdd-cpt}

In this section, we are going to prove Theorems \ref{bdd} and \ref{ess-norm}. For a positive Borel measure $\mu$ on $\C$, we use $L^2_{\alpha}(\mu)$ to denote the Hilbert space of measurable functions $f$ on $\C$ such that
$$\|f\|^2_{L^2_{\alpha}(\mu)}:=\int_{\C}|f(z)|^2e^{-\alpha|z|^2}d\mu(z)<\infty.$$
If $T_{\mu}$ is bounded on $F^2_{\alpha,w}$, then we can apply Fubini's theorem and the reproducing formula to obtain that for any $f,g\in F^2_{\alpha,w}$,
\begin{align}\label{repro}
	\langle T_{\mu}f,g\rangle_{F^2_{\alpha,w}}
	&=\int_{\C}\int_{\C}f(\xi)\overline{B^w_z(\xi)}e^{-\alpha|\xi|^2}d\mu(\xi)\overline{g(z)}e^{-\alpha|z|^2}w(z)dA(z)\nonumber\\
	&=\int_{\C}f(\xi)\overline{\int_{\C}g(z)\overline{B^w_{\xi}(z)}e^{-\alpha|z|^2}w(z)dA(z)}e^{-\alpha|\xi|^2}d\mu(\xi)\nonumber\\
	&=\int_{\C}f(\xi)\overline{g(\xi)}e^{-\alpha|\xi|^2}d\mu(\xi)\nonumber\\
	&=\langle f,g\rangle_{L^2_{\alpha}(\mu)}.
\end{align}
The following lemma indicates that if the average function $\widehat{\mu}_{w,r}$ is bounded on $\C$ for some $r>0$, then $T_{\mu}$ is densely defined on $F^2_{\alpha,w}$, and \eqref{repro} holds on a dense subset of $F^2_{\alpha,w}$.

\begin{lemma}\label{well-defined}
Let $\alpha>0$, $w\in A^{\res}_{\infty}$, and let $\mu$ be a positive Borel measure on $\C$. Suppose that $\widehat{\mu}_{w,r}$ is bounded on $\C$ for some $r>0$. Then $T_{\mu}$ is well-defined on the set of polynomials, and for any polynomials $f$ and $g$,
$$\langle T_{\mu}f,g\rangle_{F^2_{\alpha,w}}=\langle f,g\rangle_{L^2_{\alpha}(\mu)}.$$
\end{lemma}
\begin{proof}
Let $f$ be a polynomial. Then for any $z\in\C$, combining Lemma \ref{pointwise}, the inequalities \eqref{sd} and \eqref{pu} with Fubini's theorem gives that
\begin{align*}
&\int_{\C}|f(u)||B^w_z(u)|e^{-\alpha|u|^2}d\mu(u)\\
&\ \lesssim\int_{\C}\frac{1}{w(D(u,r))}\int_{D(u,r)}|f(\xi)||B^w_z(\xi)|e^{-\alpha|\xi|^2}w(\xi)dA(\xi)d\mu(u)\\
&\ \asymp\int_{\C}|f(\xi)||B^w_z(\xi)|e^{-\alpha|\xi|^2}w(\xi)\widehat{\mu}_{w,r}(\xi)dA(\xi)\\
&\ \lesssim\int_{\C}|f(\xi)||B^w_z(\xi)|e^{-\alpha|\xi|^2}w(\xi)dA(\xi)\\
&\ \lesssim\frac{e^{\frac{\alpha}{2}|z|^2}}{w(D(z,r))^{1/2}}\int_{\C}|f(\xi)|e^{-\frac{\alpha}{2}|\xi|^2}
    \frac{w(\xi)}{w(D(\xi,r))^{1/2}}dA(\xi).
\end{align*}
By \eqref{dd}, there exists $C\geq1$ such that for any $\xi\in\C$,
$$w(D(\xi,r))\gtrsim C^{-|\xi|}w(D(0,r)).$$
Consequently,
\begin{align*}
\int_{\C}|f(u)||B^w_z(u)|e^{-\alpha|u|^2}d\mu(u)
&\lesssim\frac{e^{\frac{\alpha}{2}|z|^2}}{w(D(z,r))^{1/2}}
    \int_{\C}|f(\xi)|C^{\frac{1}{2}|\xi|}e^{-\frac{\alpha}{2}|\xi|^2}w(\xi)dA(\xi)\\
&\lesssim\frac{e^{\frac{\alpha}{2}|z|^2}}{w(D(z,r))^{1/2}}\int_{\C}e^{-\frac{\alpha}{4}|\xi|^2}w(\xi)dA(\xi)\\
&\lesssim\frac{e^{\frac{\alpha}{2}|z|^2}}{w(D(z,r))^{1/2}}<\infty,
\end{align*}
where we have used \eqref{int} and the fact that $f$ is a polynomial. Hence $T_{\mu}f$ is well-defined. Similarly, for any polynomial $g$,
\begin{align*}
&\int_{\C}\int_{\C}|f(u)||B^w_z(u)|e^{-\alpha|u|^2}d\mu(u)|g(z)|e^{-\alpha|z|^2}w(z)dA(z)\\
&\ \lesssim\int_{\C}|g(z)|e^{-\frac{\alpha}{2}|z|^2}\frac{w(z)}{w(D(z,r))^{1/2}}dA(z)\\
&\ \lesssim\int_{\C}e^{-\frac{\alpha}{4}|z|^2}w(z)dA(z)<\infty.
\end{align*}
Therefore, as in \eqref{repro}, we can apply Fubini's theorem and the reproducing formula to obtain that
$$\langle T_{\mu}f,g\rangle_{F^2_{\alpha,w}}=\langle f,g\rangle_{L^2_{\alpha}(\mu)}.$$
The proof is complete.
\end{proof}

We are now ready to prove Theorems \ref{bdd} and \ref{ess-norm}.

\begin{proof}[Proof of Theorem \ref{bdd}]
Let $\delta\in(0,1)$ be the constant from Theorem \ref{pointwise-below}. The implication (a)$\Longrightarrow$(b) is clear since for any $z\in\C$, \eqref{repro} gives that
\begin{equation}\label{tildeup}
\widetilde{\mu}(z)=\|b^w_z\|^2_{L^2_{\alpha}(\mu)}=\langle T_{\mu}b^w_z,b^w_z\rangle_{F^2_{\alpha,w}}\leq\|T_{\mu}b^w_z\|_{F^2_{\alpha,w}}
\leq\|T_{\mu}\|_{F^2_{\alpha,w}\to F^2_{\alpha,w}}.
\end{equation}

Suppose now that (b) holds. Then for any $r\in(0,\delta)$, Theorem \ref{pointwise-below} together with Lemma \ref{norm} yields that
\begin{equation}\label{tilde>hat}
\widetilde{\mu}(z)\geq\int_{D(z,r)}\frac{|B^w_z(u)|^2}{\|B^w_z\|^2_{F^2_{\alpha,w}}}e^{-\alpha|u|^2}d\mu(u)
\gtrsim\frac{\mu(D(z,r))}{w(D(z,r))}=\widehat{\mu}_{w,r}(z).
\end{equation}
Hence (c) holds and $\sup_{\C}\widehat{\mu}_{w,r}\lesssim\sup_{\C}\widetilde{\mu}$.

Suppose next that (c) holds, i.e. $\widehat{\mu}_{w,r}$ is bounded on $\C$ for some $r\in(0,\delta)$. Then by \cite[Theorem 1.2]{CFP}, the embedding $I_d:F^2_{\alpha,w}\to L^2_{\alpha}(\mu)$ is bounded, and
$$\|I_{d}\|_{F^2_{\alpha,w}\to L^2_{\alpha}(\mu)}\asymp\left(\sup_{z\in\C}\widehat{\mu}_{w,r}(z)\right)^{\frac{1}{2}}.$$
Therefore, for any two polynomials $f$ and $g$, Lemma \ref{well-defined} together with the Cauchy--Schwarz inequality yields that
$$|\langle T_{\mu}f,g\rangle_{F^2_{\alpha,w}}|=|\langle f,g\rangle_{L^2_{\alpha}(\mu)}|
\leq\|f\|_{L^2_{\alpha}(\mu)}\|g\|_{L^2_{\alpha}(\mu)}
\lesssim\sup_{z\in\C}\widehat{\mu}_{w,r}(z)\cdot\|f\|_{F^2_{\alpha,w}}\|g\|_{F^2_{\alpha,w}}.$$
Since polynomials are dense in $F^2_{\alpha,w}$ by Lemma \ref{dense}, we conclude that $T_{\mu}$ is bounded on $F^2_{\alpha,w}$, and $\|T_{\mu}\|_{F^2_{\alpha,w}\to F^2_{\alpha,w}}\lesssim\sup_{z\in\C}\widehat{\mu}_{w,r}(z)$. Hence (a) holds and the proof is finished.
\end{proof}

\begin{proof}[Proof of Theorem \ref{ess-norm}]
Let $\delta\in(0,1)$ be the constant from Theorem \ref{pointwise-below}, and fix $r\in(0,\delta)$. By \eqref{tilde>hat}, we have
$$\limsup_{|z|\to\infty}\widehat{\mu}_{w,r}(z)\lesssim\limsup_{|z|\to\infty}\widetilde{\mu}(z).$$
Therefore, it is sufficient to show
\begin{equation}\label{suff}
\limsup_{|z|\to\infty}\widetilde{\mu}(z)\lesssim\|T_{\mu}\|_{e,F^2_{\alpha,w}\to F^2_{\alpha,w}}\lesssim
\limsup_{|z|\to\infty}\widehat{\mu}_{w,r}(z).
\end{equation}
We start with the first estimate. Let $K$ be a compact operator on $F^2_{\alpha,w}$. Since Proposition \ref{weakly0} says that the normalized reproducing kernel $b^w_z$ converges to $0$ weakly as $|z|\to\infty$, we have $Kb^w_z\to0$ in $F^2_{\alpha,w}$ as $|z|\to\infty$. Therefore, we deduce from \eqref{tildeup} that
\begin{align*}
\|T_{\mu}-K\|_{F^2_{\alpha,w}\to F^2_{\alpha,w}}
&\geq\limsup_{|z|\to\infty}\|(T_{\mu}-K)b^w_z\|_{F^2_{\alpha,w}}\\
&\geq\limsup_{|z|\to\infty}\left(\|T_{\mu}b^w_z\|_{F^2_{\alpha,w}}-\|Kb^w_z\|_{F^2_{\alpha,w}}\right)\\
&=\limsup_{|z|\to\infty}\|T_{\mu}b^w_z\|_{F^2_{\alpha,w}}\\
&\geq\limsup_{|z|\to\infty}\widetilde{\mu}(z).
\end{align*}
Since $K\in\mathcal{K}(F^2_{\alpha,w})$ is arbitrary, we obtain that
$$\|T_{\mu}\|_{e,F^2_{\alpha,w}\to F^2_{\alpha,w}}\geq\limsup_{|z|\to\infty}\widetilde{\mu}(z).$$

We next concentrate on the second estimate of \eqref{suff}. Assume that $\{f_j\}$ is an orthonormal basis of $F^2_{\alpha,w}$. For each positive integer $n$, let the operator $Q_n$ be defined by
$$Q_nf:=\sum_{j=1}^n\langle f,f_j\rangle_{F^2_{\alpha,w}}f_j,\quad f\in F^2_{\alpha,w}.$$
Then $Q_n$ is compact on $F^2_{\alpha,w}$. Writing $R_n=I-Q_n$, we have
$$R_nT_{\mu}R_n=T_{\mu}-T_{\mu}Q_n-Q_nT_{\mu}+Q_nT_{\mu}Q_n.$$
Consequently, for each positive integer $n$,
$$\|T_{\mu}\|_{e,F^2_{\alpha,w}\to F^2_{\alpha,w}}\leq\|R_nT_{\mu}R_n\|_{F^2_{\alpha,w}\to F^2_{\alpha,w}}.$$
We now claim that for any $t>r$,
\begin{equation}\label{claim}
\limsup_{n\to\infty}\sup_{\|f\|_{F^2_{\alpha,w}}=1}\|R_nf\|^2_{L^2_{\alpha}(\mu)}
\lesssim\sup_{z\in\C\setminus D(0,t-r)}\widehat{\mu}_{w,r}(z).
\end{equation}
Then, noting that $R_n$ is self-adjoint, we apply \eqref{repro}, the Cauchy--Schwarz inequality and \eqref{claim} to deduce that
\begin{align*}
\|T_{\mu}\|_{e,F^2_{\alpha,w}\to F^2_{\alpha,w}}
&\leq\limsup_{n\to\infty}\|R_nT_{\mu}R_n\|_{F^2_{\alpha,w}\to F^2_{\alpha,w}}\\
&=\limsup_{n\to\infty}\sup_{\|f\|_{F^2_{\alpha,w}}=\|g\|_{F^2_{\alpha,w}}=1}|\langle R_nT_{\mu}R_nf,g\rangle_{F^2_{\alpha,w}}|\\
&=\limsup_{n\to\infty}\sup_{\|f\|_{F^2_{\alpha,w}}=\|g\|_{F^2_{\alpha,w}}=1}|\langle T_{\mu}R_nf,R_ng\rangle_{F^2_{\alpha,w}}|\\
&=\limsup_{n\to\infty}\sup_{\|f\|_{F^2_{\alpha,w}}=\|g\|_{F^2_{\alpha,w}}=1}|\langle R_nf,R_ng\rangle_{L^2_{\alpha}(\mu)}|\\
&\leq\limsup_{n\to\infty}\sup_{\|f\|_{F^2_{\alpha,w}}=1}\|R_nf\|^2_{L^2_{\alpha}(\mu)}\\
&\lesssim\sup_{z\in\C\setminus D(0,t-r)}\widehat{\mu}_{w,r}(z).
\end{align*}
Letting $t\to\infty$, we obtain that
$$\|T_{\mu}\|_{e,F^2_{\alpha,w}\to F^2_{\alpha,w}}\lesssim\limsup_{|z|\to\infty}\widehat{\mu}_{w,r}(z).$$
It remains to establish \eqref{claim}. Fix $t>r$. For any $f\in F^2_{\alpha,w}$ with $\|f\|_{F^2_{\alpha,w}}=1$ and $z\in\C$,
$$|R_nf(z)|^2=|\langle R_nf,B^w_z\rangle_{F^2_{\alpha,w}}|^2=|\langle f,R_nB^w_z\rangle_{F^2_{\alpha,w}}|^2
\leq\|R_nB^w_z\|^2_{F^2_{\alpha,w}},$$
which implies that
$$\sup_{\|f\|_{F^2_{\alpha,w}}=1}\int_{D(0,t)}|R_nf(z)|^2e^{-\alpha|z|^2}d\mu(z)
\leq\int_{D(0,t)}\|R_nB^w_z\|^2_{F^2_{\alpha,w}}e^{-\alpha|z|^2}d\mu(z).$$
Note that the boundedness of $T_{\mu}$ on $F^2_{\alpha,w}$ implies that $\mu(D(0,t))<\infty$. Since for any $z\in\C$,
$$\lim_{n\to\infty}\|R_nB^w_z\|_{F^2_{\alpha,w}}=0,$$
and by Lemma \ref{norm},
$$\|R_nB^w_z\|^2_{F^2_{\alpha,w}}e^{-\alpha|z|^2}\leq\|B^w_z\|^2_{F^2_{\alpha,w}}e^{-\alpha|z|^2}\lesssim w(D(z,1))^{-1},$$
which is bounded on $D(0,t)$, we may apply the dominated convergence theorem to deduce that
$$\lim_{n\to\infty}\int_{D(0,t)}\|R_nB^w_z\|^2_{F^2_{\alpha,w}}e^{-\alpha|z|^2}d\mu(z)=0.$$
Therefore,
\begin{equation}\label{in-disk}
\lim_{n\to\infty}\sup_{\|f\|_{F^2_{\alpha,w}}=1}\int_{D(0,t)}|R_nf(z)|^2e^{-\alpha|z|^2}d\mu(z)=0.
\end{equation}
On the other hand, by Lemma \ref{pointwise}, the inequality \eqref{sd} and Fubini's theorem, we have for any $f\in F^2_{\alpha,w}$ with $\|f\|_{F^2_{\alpha,w}}=1$,
\begin{align*}
&\int_{\C\setminus D(0,t)}|R_nf(z)|^2e^{-\alpha|z|^2}d\mu(z)\\
&\ \lesssim\int_{\C\setminus D(0,t)}\frac{1}{w(D(z,r))}\int_{D(z,r)}|R_nf(\xi)|^2e^{-\alpha|\xi|^2}w(\xi)dA(\xi)d\mu(z)\\
&\ =\int_{\C}|R_nf(\xi)|^2e^{-\alpha|\xi|^2}w(\xi)\int_{D(\xi,r)\cap(\C\setminus D(0,t))}\frac{d\mu(z)}{w(D(z,r))}dA(\xi)\\
&\ \lesssim\int_{\C\setminus D(0,t-r)}|R_nf(\xi)|^2e^{-\alpha|\xi|^2}w(\xi)\widehat{\mu}_{w,r}(\xi)dA(\xi)\\
&\ \leq\sup_{z\in \C\setminus D(0,t-r)}\widehat{\mu}_{w,r}(z)\cdot\int_{\C\setminus D(0,t-r)}|R_nf(\xi)|^2e^{-\alpha|\xi|^2}w(\xi)dA(\xi)\\
&\ \leq\sup_{z\in \C\setminus D(0,t-r)}\widehat{\mu}_{w,r}(z).
\end{align*}
This, together with \eqref{in-disk}, establishes \eqref{claim} and finishes the proof.
\end{proof}

\section{Schatten $p$-class Toeplitz operators}\label{sp}


The purpose of this section is to prove Theorem \ref{schatten-p}. We will actually establish a more general result that contains Theorem \ref{schatten-p} as a special case.

Let $T$ be a compact operator on a separable Hilbert space $H$ with singular value sequence $\{s_n(T)\}$, and let $h:\R^+\to\R^+$ be a continuous increasing function such that $h(0)=0$. Following \cite{EMMN}, we say that $T\in \s_h(H)$ if there exists $C>0$ such that
$$\sum_{n\geq1}h(Cs_n(T))<\infty.$$
The following theorem characterizes the $\s_h$-class Toeplitz operators on $F^2_{\alpha,w}$ for convex functions $h$, which reduces to Theorem \ref{schatten-p} when $h(t)=t^p$ for $p\geq1$.

\begin{theorem}\label{schatten-h}
Let $\alpha>0$, $w\in A^{\res}_{\infty}$, and let $\mu$ be a positive Borel measure on $\C$. Suppose that $h:\R^+\to \R^+$ is a continuous increasing convex function such that $h(0)=0$, and let $\delta\in(0,1)$ be the constant from Theorem \ref{pointwise-below}. Then the following conditions are equivalent:
\begin{enumerate}
	\item [(a)] $T_{\mu}\in\s_h(F^2_{\alpha,w})$;
	\item [(b)] there exists $C>0$ such that
	$$\int_{\C}h\left(C\widetilde{\mu}(z)\right)dA(z)<\infty;$$
	\item [(c)] there exists $C>0$ such that for some (or any) $r\in (0,\delta)$,
	$$\int_{\C}h\left(C\widehat{\mu}_{w,r}(z)\right)dA(z)<\infty.$$
\end{enumerate}
Moreover, there exist $C_1,C_2,C_3>0$ such that
$$\sum_{n\geq1}h\left(C_1s_n(T_{\mu})\right)\asymp\int_{\C}h\left(C_2\widetilde{\mu}(z)\right)dA(z)
\asymp\int_{\C}h\left(C_3\widehat{\mu}_{w,r}(z)\right)dA(z).$$
\end{theorem}
\begin{proof}
By Corollary \ref{cpt}, it is clear that if one of (a), (b) and (c) holds, then $T_{\mu}$ is compact on $F^2_{\alpha,w}$. Since $T_{\mu}$ is self-adjoint, we may assume its canonical decomposition is given by
\begin{equation}\label{cd}
T_{\mu}f=\sum_{n\geq1}s_n\langle f,f_n\rangle_{F^2_{\alpha,w}}f_n,\quad f\in F^2_{\alpha,w},
\end{equation}
where $\{f_n\}$ is an orthonormal set of $F^2_{\alpha,w}$.

Suppose first that $(a)$ holds. That is,
$$\sum_{n\geq1}h(Cs_n)<\infty$$
for some $C>0$. Noting that for any $z\in\C$, $\sum_{n\geq1}\left|\langle b^w_z,f_n\rangle_{F^2_{\alpha,w}}\right|^2\leq1$, by the definition of $\widetilde{\mu}$, \eqref{repro}, \eqref{cd} and Jensen’s inequality, we have
\begin{align*}
h\left(C\widetilde{\mu}(z)\right)
&=h\left(C\langle T_{\mu}b^w_z,b^w_z\rangle_{F^2_{\alpha,w}}\right)\\
&=h\left(C\sum_{n\geq1}s_n\left|\langle b^w_z,f_n\rangle_{F^2_{\alpha,w}}\right|^2\right)\\
&\leq\sum_{n\geq1}h\left(Cs_n\right)\left|\langle b^w_z,f_n\rangle_{F^2_{\alpha,w}}\right|^2\\
&=\sum_{n\geq1}h\left(Cs_n\right)|f_n(z)|^2\|B^w_z\|^{-2}_{F^2_{\alpha,w}}.
\end{align*}
Therefore, applying Lemmas \ref{norm} and \ref{hat-norm}, we establish that
\begin{align*}
\int_{\C}h\left(C\widetilde{\mu}(z)\right)dA(z)
&\leq\sum_{n\geq1}h\left(Cs_n\right)\int_{\C}|f_n(z)|^2\|B^w_z\|^{-2}_{F^2_{\alpha,w}}dA(z)\\
&\asymp\sum_{n\geq1}h\left(Cs_n\right)\int_{\C}|f_n(z)|^2e^{-\alpha|z|^2}w(D(z,1))dA(z)\\
&\asymp\sum_{n\geq1}h\left(Cs_n\right)<\infty,
\end{align*}
and consequently, (b) holds.

The implication (b)$\Longrightarrow$(c) follows from \eqref{tilde>hat}.

Suppose next that (c) holds. That is, there exists $r\in(0,\delta)$ and $C'>0$ such that
$$\int_{\C}h\left(C'\widehat{\mu}_{w,r}(z)\right)dA(z)<\infty.$$
By \eqref{cd} and \eqref{repro}, we have for any $n\geq1$,
$$s_n=\langle T_{\mu}f_n,f_n\rangle_{F^2_{\alpha,w}}=\|f\|_{L^2_{\alpha}(\mu)}^2,$$
which, together with Lemma \ref{pointwise}, the inequality \eqref{sd} and Fubini's theorem, implies that for any $c>0$,
\begin{align*}
h(cs_n)
&=h\left(c\int_{\C}|f_n(\xi)|^2e^{-\alpha|\xi|^2}d\mu(\xi)\right)\\
&\leq h\left(c_1\int_{\C}\frac{1}{w(D(\xi,r))}\int_{D(\xi,r)}|f_n(z)|^2e^{-\alpha|z|^2}\widehat{w}_r(z)dA(z)d\mu(\xi)\right)\\
&\leq h\left(c_2\int_{\C}|f_n(z)|^2e^{-\alpha|z|^2}\widehat{w}_r(z)\widehat{\mu}_{w,r}(z)dA(z)\right).
\end{align*}
By Lemma \ref{hat-norm},
$$\int_{\C}|f_n(z)|^2e^{-\alpha|z|^2}\widehat{w}_r(z)dA(z)\asymp1.$$
Hence we can choose some $c>0$ and use Jensen's inequality to obtain that
$$h(cs_n)\leq\int_{\C}h\left(C'\widehat{\mu}_{w,r}(z)\right)|f_n(z)|^2e^{-\alpha|z|^2}\widehat{w}_r(z)dA(z),$$
which, combined with Lemma \ref{norm}, implies that
\begin{align*}
\sum_{n\geq1}h\left(cs_n\right)
&\leq\sum_{n\geq1}\int_{\C}h\left(C'\widehat{\mu}_{w,r}(z)\right)|f_n(z)|^2e^{-\alpha|z|^2}\widehat{w}_r(z)dA(z)\\
&=\int_{\C}h\left(C'\widehat{\mu}_{w,r}(z)\right)\left(\sum_{n\geq1}|\langle B^w_z,f_n\rangle_{F^2_{\alpha,w}}|^2\right)
    e^{-\alpha|z|^2}\widehat{w}_r(z)dA(z)\\
&\leq\int_{\C}h\left(C'\widehat{\mu}_{w,r}(z)\right)\|B^w_z\|^2_{F^2_{\alpha,w}}e^{-\alpha|z|^2}\widehat{w}_r(z)dA(z)\\
&\asymp\int_{\C}h\left(C'\widehat{\mu}_{w,r}(z)\right)dA(z)<\infty.
\end{align*}
This establishes (a) and finishes the proof.
\end{proof}

As an application of Theorem \ref{schatten-h}, we can estimate the decay of singular values of compact Toeplitz operators $T_{\mu}$ on $F^2_{\alpha,w}$. Let $\{s_n(T_{\mu})\}$ be the singular value sequence of $T_{\mu}$. The following corollary follows from Theorem \ref{schatten-h} and \cite[Lemma 6.1]{EMMN} directly.

\begin{corollary}\label{decay}
Let $\alpha>0$, $w\in A^{\res}_{\infty}$, and let $\mu$ be a positive Borel measure on $\C$ such that $T_{\mu}$ is compact on $F^2_{\alpha,w}$. Let $\delta\in(0,1)$ be the constant from Theorem \ref{pointwise-below}. Suppose that $\eta:[1,+\infty)\to (0,+\infty)$ is a continuous decreasing function such that $\eta(+\infty)=0$ and
$$\eta(t\log t)\asymp\eta(t),\quad t\to+\infty.$$
If, in addition, the function $h_{\eta}:\R^+\to\R^+$ defined by $h_{\eta}(\eta(t))=1/t$ is convex, then the following conditions are equivalent:
\begin{enumerate}
	\item [(a)] $s_n(T_{\mu})\lesssim \eta(n)$;
	\item [(b)] there exists $C>0$ such that
	$$\int_{\C}h_{\eta}\left(C\widetilde{\mu}(z)\right)dA(z)<\infty;$$
	\item [(c)] there exists $C>0$ such that for some (or any) $r\in (0,\delta)$,
	$$\int_{\C}h_{\eta}\left(C\widehat{\mu}_{w,r}(z)\right)dA(z)<\infty.$$
\end{enumerate}
\end{corollary}

\begin{example}
Let $\alpha>0$, $w\in A^{\res}_{\infty}$, and let $\mu$ be a positive Borel measure on $\C$ such that $T_{\mu}$ is compact on $F^2_{\alpha,w}$. Suppose that $s_n(T_{\mu})$ is the $n$th singular value of $T_{\mu}$, and let $\delta\in(0,1)$ be the constant from Theorem \ref{pointwise-below}. Then for any $\gamma>0$, the following conditions are equivalent:
\begin{enumerate}
	\item [(a)] $s_n(T_{\mu})\lesssim(\log n)^{-\gamma}$;
	\item [(b)] for some $C>0$,
	$$\int_{\C}\exp\left(-\left(C\widetilde{\mu}(z)\right)^{-\frac{1}{\gamma}}\right)dA(z)<\infty;$$
	\item [(c)] for some $r\in(0,\delta)$ and $C>0$, $$\int_{\C}\exp\left(-\left(C\widehat{\mu}_{w,r}(z)\right)^{-\frac{1}{\gamma}}\right)dA(z)<\infty.$$
\end{enumerate}
In fact, let
$$\eta(t)=\frac{1}{(1+\gamma+\log t)^{\gamma}},\quad t\in[1,+\infty)$$
and
\begin{equation*}
h_{\eta}(t)=
\begin{cases}
	0,&   \text{if $t=0$},\\
	\exp\left(1+\gamma-t^{-\frac{1}{\gamma}}\right),&   \text{if $0<t\leq(1+\gamma)^{-\gamma}$},\\
	\frac{1}{\gamma}(1+\gamma)^{1+\gamma}t-\frac{1}{\gamma},& \text{if $t>(1+\gamma)^{-\gamma}$}.
\end{cases}
\end{equation*}
Then $h_{\eta}$ is convex, $\eta(t\log t)\asymp\eta(t)$ as $t\to+\infty$, and for any $t\in[1,+\infty)$, $h_{\eta}(\eta(t))=1/t$. Hence by Corollary \ref{decay}, $s_n(T_{\mu})\lesssim (\log n)^{-\gamma}$ if and only if for some $C>0$,
$$\int_{\C}h_{\eta}\left(C\widetilde{\mu}(z)\right)dA(z)<\infty.$$
Since the compactness of $T_{\mu}$ implies that $\widetilde{\mu}$ is bounded on $\C$ and vanishes at infinity, we may choose $R>0$ such that $C\widetilde{\mu}(z)<(1+\gamma)^{-\gamma}$ whenever $|z|\geq R$. Noting that
$$\int_{D(0,R)}h_{\eta}\left(C\widetilde{\mu}(z)\right)dA(z)<\infty$$
and
$$\int_{D(0,R)}\exp\left(-(C\widetilde{\mu}(z))^{-\frac{1}{\gamma}}\right)dA(z)<\infty,$$
we obtain that (a) and (b) are equivalent. The equivalence of (a) and (c) is similar.
\end{example}

\section{Applications}\label{app}

In this section, we give some applications of the main results to Volterra operators and weighted composition operators.

\subsection{Volterra operators}

Given an entire function $g$ on $\C$, the Volterra operator $J_g$ is defined for entire functions $f$ by
$$J_gf(z):=\int_0^zf(\xi)g'(\xi)d\xi,\quad z\in\C.$$
It follows from \cite[Theorems 3.1 and 3.2]{Xu} that, for $\alpha,p>0$ and $w\in A^{\res}_{\infty}$, $J_g$ is bounded (resp. compact) on $F^p_{\alpha,w}$ if and only if $g$ is a polynomial of degree not more than $2$ (resp. not more than $1$). We here apply Theorem \ref{schatten-p} to characterize the Schatten $p$-class Volterra operators on $F^2_{\alpha,w}$. To this end, define the integral pairing $\langle \cdot,\cdot\rangle_{*}$ as follows:
$$\langle f,g\rangle_{*}:=f(0)\overline{g(0)}+\int_{\C}f'(z)\overline{g'(z)}e^{-\alpha|z|^2}\frac{\widehat{w}_1(z)}{(1+|z|)^2}dA(z).$$

\begin{corollary}
	Let $\alpha>0$, $w\in A^{\res}_{\infty}$, and let $g(z)=az+b$ for some $a,b\in\C$. Then $J_g$ belongs to $\s_p(F^2_{\alpha,w})$ for all $p>2$, but it fails to be Hilbert--Schmidt unless $a=0$.
\end{corollary}
\begin{proof}
	By Lemma \ref{hat-norm} and \cite[Theorem 1.1]{CFP}, the pairing $\langle \cdot,\cdot\rangle_{*}$ is an inner product on $F^2_{\alpha,w}$ that induces an equivalent norm. For any $f,h\in F^2_{\alpha,w}$, it is clear that
	$$\langle J_g^*J_gf,h\rangle_{*}=\int_{\C}f(z)\overline{h(z)}|g'(z)|^2e^{-\alpha|z|^2}\frac{\widehat{w}_1(z)}{(1+|z|)^2}dA(z)
	=\langle f,h\rangle_{L^2_{\alpha}(\mu_g)},$$
	where the measure $\mu_g$ is defined by
	$$d\mu_g(z):=|g'(z)|^2\frac{\widehat{w}_1(z)}{(1+|z|)^2}dA(z).$$
	Consequently, by \eqref{repro}, $J_g^*J_g=T_{\mu_g}$, which implies that $J_g\in\s_p(F^2_{\alpha,w})$ if and only if $T_{\mu_g}\in\s_{p/2}(F^2_{\alpha,w})$. This together with Theorem \ref{schatten-p} gives that for $p\geq2$, $J_g\in\s_p(F^2_{\alpha,w})$ if and only if $\widehat{(\mu_g)}_{w,r}\in L^{p/2}(\C,dA)$. Note that for any $z\in\C$, \eqref{sd} yields 
	$$\widehat{(\mu_g)}_{w,r}(z)=\frac{1}{w(D(z,r))}\int_{D(z,r)}|g'(\xi)|^2\frac{w(D(\xi,1))}{(1+|\xi|)^2}dA(\xi)
	\asymp\frac{|a|^2}{(1+|z|)^2}.$$
	The desired result follows easily.
\end{proof}

\subsection{Weighted composition operators}

Given two entire functions $\varphi,\psi$ on $\C$, the weighted composition operators $W_{\varphi,\psi}$ is defined by
$$W_{\varphi,\psi}f:=\psi\cdot f\circ\varphi.$$
Let $\mu_{\varphi,\psi}$ denote the positive pull-back measure on $\C$ defined by
$$\mu_{\varphi,\psi}(E):=\int_{\varphi^{-1}(E)}|\psi(z)|^2e^{-\alpha(|z|^2-|\varphi(z)|^2)}w(z)dA(z)$$
for every Borel subset $E$ of $\C$. Then for any $f,g\in F^2_{\alpha,w}$,
\begin{align*}
\langle W_{\varphi,\psi}^*W_{\varphi,\psi}f,g\rangle_{F^2_{\alpha,w}}
&=\int_{\C}f(\varphi(z))\overline{g(\varphi(z))}|\psi(z)|^2e^{-\alpha|z|^2}w(z)dA(z)\\
&=\int_{\C}f(z)\overline{g(z)}e^{-\alpha|z|^2}d\mu_{\varphi,\psi}(z).
\end{align*}
Therefore,  $W^*_{\varphi,\psi}W_{\varphi,\psi}=T_{\mu_{\varphi,\psi}}$ and the following result is a direct consequence of Theorems \ref{bdd}, \ref{schatten-p} and Corollary \ref{cpt}. 

\begin{corollary}
Let $\alpha>0$, $w\in A^{\res}_{\infty}$, and let $\varphi,\psi$ be entire functions on $\C$. Then there exists $\delta\in(0,1)$ such that
\begin{enumerate}
	\item [(1)] $W_{\varphi,\psi}$ is bounded on $F^2_{\alpha,w}$ if and only if for some (or any) $r\in(0,\delta)$,
	$$\sup_{z\in\C}\frac{1}{w(D(z,r))}\int_{\varphi^{-1}(D(z,r))}|\psi(\xi)|^2e^{-\alpha(|\xi|^2-|\varphi(\xi)|^2)}w(\xi)dA(\xi)<\infty;$$
	\item [(2)] $W_{\varphi,\psi}$ is compact on $F^2_{\alpha,w}$ if and only if for some (or any) $r\in(0,\delta)$,
	$$\lim_{|z|\to\infty}\frac{1}{w(D(z,r))}\int_{\varphi^{-1}(D(z,r))}|\psi(\xi)|^2e^{-\alpha(|\xi|^2-|\varphi(\xi)|^2)}w(\xi)dA(\xi)=0;$$
	\item [(3)] for $p\geq2$, $W_{\varphi,\psi}\in \s_p(F^2_{\alpha,w})$ if and only if for some (or any) $r\in(0,\delta)$, the function
	$$z\mapsto\frac{1}{w(D(z,r))}\int_{\varphi^{-1}(D(z,r))}|\psi(\xi)|^2e^{-\alpha(|\xi|^2-|\varphi(\xi)|^2)}w(\xi)dA(\xi)$$
	belongs to $L^{p/2}(\C,dA)$.
\end{enumerate}
\end{corollary}

\medskip





\end{document}